\documentclass{article}

\usepackage{amsthm}
\usepackage{amsmath}
\usepackage{amssymb}
\usepackage{cite}
\usepackage{color}
\usepackage{multirow}
\usepackage{graphicx}
\usepackage{booktabs}
\usepackage{scalefnt}
\usepackage{lscape}
\usepackage{url}
\usepackage{bm}

\usepackage{algorithmic}
\usepackage[section]{algorithm}

\DeclareMathOperator*{\grad}{grad}

\DeclareMathOperator*{\Hess}{Hess}

\newcommand{\real}{\mathbb{R}}
\newcommand{\nat}{\mathbb{N}}

\newcommand{\abs}[1]{\left|{#1}\right|}
\newcommand{\norm}[1]{\left\lVert{#1}\right\rVert}
\newcommand{\ip}[2]{\left\langle{#1,#2}\right\rangle}
\newcommand{\tran}{\mathcal{T}}

\newtheorem{definition}{Definition}[section]
\newtheorem{theorem}{Theorem}[section]

\newtheorem{lemma}{Lemma}[section]

\newtheorem{problem}{Problem}[section]
\newtheorem{assumption}{Assumption}[section]

\title{Global Convergence of Hager--Zhang type Riemannian Conjugate Gradient Method}
\author{Hiroyuki Sakai, Hiroyuki Sato, and Hideaki Iiduka}
\date{}

\begin{document}
\maketitle

\begin{abstract}
This paper presents the Hager--Zhang (HZ)-type Riemannian conjugate gradient method that uses the exponential retraction.
We also present global convergence analyses of our proposed method under two kinds of assumptions.
Moreover, we numerically compare our proposed methods with the existing methods by solving two kinds of Riemannian optimization problems on the unit sphere.
The numerical results show that our proposed method has much better performance than the existing methods, i.e., the FR, DY, PRP and HS methods.
In particular, they show that it has much higher performance than existing methods including the hybrid ones in computing the stability number of graphs problem.
\end{abstract}

\section{Introduction}\label{sec:1}
Riemannian optimization has been widely researched along with the developments of real-world applications in various fields,
such as natural language processing \cite{NIPS2017_59dfa2df,9339934}, signal processing \cite{doi:10.1137/110845768},
and computer vision \cite{NIPS2015_dbe272ba,pmlr-v48-kasai16}, in which large-scale problems can be expressed as certain optimization problems on Riemannian manifolds.

Many useful gradient methods \cite{absil2008optimization,sato2021} have been developed for Riemannian optimization that can be obtained by extending the existing methods in Euclidean space to a Riemannian manifold. 
However, such extension is not always easy.
For example, in the Euclidean space setting, the $(k+1)$-th approximation of optimal solutions is $\bm{x}_{k+1} = \bm{x}_k + \alpha_k \bm{\eta}_k$,
where $\alpha_k > 0$, $\bm{x}_k \in \mathbb{R}^n$ is a point at the $k$-th iteration, and $\bm{\eta}_k$ is the search direction. 
However, such an update cannot be defined for general Riemannian manifolds
because of nonlinear Riemannian geometric structure.
We can generalize Riemannian gradient methods using {\em retractions} and {\em transports} that are divided into various types, as described below.

Smith \cite{smith1994} proposed using {\em exponential retraction} and {\em parallel transport} to generalize the optimization methods from Euclidean space to a Riemannian manifold. 
Absil, Mahony, and Sepulchre \cite{absil2008optimization} proposed using a {\em general retraction} that approximates the exponential retraction
and a {\em vector transport} which approximates the parallel transport.
Note that a general retraction (resp. vector transport) is a generalization of the exponential retraction (resp. parallel transport).

We focus on {\em Riemannian conjugate gradient (RCG) methods} as they offer both theoretical and practical benefits.
A theoretical benefit of RCG methods is that we can show that they generate {\em sufficient descent search directions},
which decrease an objective function at every iteration, and converge {\em globally}, i.e., without depending on the choice of the initial point. 
A practical benefit of RCG methods is that they have efficient numerical performances, as shown in the previous studies \cite{absil2008optimization,sato2021}.

\subsection{Previous results}\label{subsec:1.1}
The results for RCG methods that satisfy the sufficient descent condition and global convergence are summarized as in Table \ref{table:1}.

Ring and Wirth \cite{doi:10.1137/11082885X} presented a Fletcher--Reeves (FR) type of RCG method using a general retraction and vector transport,
which is defined by the differentiated retraction, under the strong Wolfe conditions.
The vector transport they used in \cite{doi:10.1137/11082885X} is assumed not to increase the norm of the search direction vector,
which would be unnatural in both theory and practice.
To overcome this limitations, Sato and Iwai \cite{sato2013new} defined a {\em scaled vector transport} and showed convergence of the FR-type RCG method using a general retraction and scaled vector transport.   

Sato \cite{sato2015dai} also investigated a Dai--Yuan (DY) type of RCG method using a general retraction and scaled vector transport and showed that
it generates a sufficient descent direction and converges globally under the Wolfe conditions (``DY" row in Table \ref{table:1}).
Comparison of the results in \cite{sato2013new} with those in \cite{sato2015dai} reveals that
the DY-type RCG method has a better global convergence than the FR-type one because it is based on the assumption of the Wolfe conditions, which are weaker than the strong Wolfe conditions.

A recently introduced hybrid RCG method \cite{sakai2020hybrid} is defined by combining the good global convergence of the DY-type RCG method (see description above) with the efficient numerical performance of a Hestenes--Stiefel (HS) type of RCG method. 
This hybrid method generates a sufficient descent direction and converges globally under the strong Wolfe conditions (``HS-DY hybrid" row in Table~\ref{table:1}).
Another recently introduced hybrid method \cite{sakai2020sufficient} combines the FR-tysspe RCG method with a Polak--Ribi\`ere--Polyak (PRP) type of RCG method.
This hybrid method also generates a sufficient descent direction and converges globally under the strong Wolfe conditions (``FR-PRP hybrid" row in Table~\ref{table:1}).

\subsection{Goals}\label{subsec:1.2}
As described in Section \ref{subsec:1.1}, and shown in Table \ref{table:1}, existing RCG methods are capable for solving Riemannian optimization problems.
Nevertheless, there are other powerful conjugate gradient methods in Euclidean space that could be generalized to Riemannian manifolds. 
A particularly interesting Euclidean conjugate gradient (ECG) method is the Hager--Zhang (HZ) type\footnote{\url{http://users.clas.ufl.edu/hager/papers/Software/}} \cite{doi:10.1137/030601880} of conjugate gradient method, which is a very efficient conjugate gradient method for Euclidean optimization.
Accordingly, the first goal of this paper 
is to clarify whether or not the HZ-type ECG method can be theoretically extended to a Riemannian manifold so as to guarantee its global convergence.
Sakai and Iiduka \cite{sakai2020sufficient} showed that the HZ-type RCG method using a general retraction and scaled vector transport generates a sufficient descent direction (``HZ" row in Table~\ref{table:1}).
This sufficient descent property does not depend on the line search conditions.
However, the global convergence of the HZ-type RCG method has not been determined.
 
The second goal is to determine whether that the HZ-type RCG method performs better than the existing RCG methods listed in Table \ref{table:1}.  
The HZ-type ECG method tends to perform better in the Euclidean space setting than other ECG methods.
Therefore, it would be useful to know whether the HZ-type RCG method
has the same performance as the HZ-type ECG method.

\subsection{Contributions}
This paper makes two contributions. 
The first contribution is to show that the HZ-type RCG method using the {\em exponential retraction} and vector transport converges globally under the Wolfe conditions (Theorem~\ref{thm:converge1}).
This contribution is an extension of Theorem 2.2 in \cite{doi:10.1137/030601880} to a Riemannian manifold and shows theoretically for the first time the global convergence of the HZ-type RCG method. 
The second contribution is to provide numerical comparisons of the HZ-type RCG method with the existing RCG methods. 
The numerical results of this paper indicate that the HZ-type RCG method performs better than the existing ones in computing the stability number of graphs problem.

\subsection{Difficulty to prove Theorem \ref{thm:converge1}}
A way to guarantee the global convergence property of the HZ-type ECG method in the Euclidean space is to assume the strong convexity of the objective function $f$.
We thus assume that there exists a constant $\mu > 0$ such that
\begin{equation}
\label{eq:E_strong_convexity}
    (\nabla f(\bm{x}) - \nabla f(\bm{y}))^\top (\bm{x}-\bm{y}) \ge \mu \norm{\bm{x}-\bm{y}}^2
\end{equation}
holds for any $\bm{x}, \bm{y} \in \real^n$, which is equivalent to the condition that the smallest eigenvalue of the Hessian $\nabla^2 f(\bm{x})$ for any $\bm{x} \in \real^n$ is not less than $\mu$.
Expression~\eqref{eq:E_strong_convexity} is more useful in convergence analysis.

However, \eqref{eq:E_strong_convexity} cannot be directly generalized to the Riemannian case.
Instead, a natural definition of the strong convexity of $f$ on a Riemannian manifold $M$ is that there exists a constant $\mu > 0$ such that, for any $x \in M$, the smallest eigenvalue of the Riemannian Hessian $\Hess f(x)$ is not less than $\mu$.
In Theorem \ref{thm:converge1}, we have to start with this condition and without a Riemannian counterpart of \eqref{eq:E_strong_convexity}.

Noting that \eqref{eq:E_strong_convexity} is used to prove $(\nabla f(\bm{x}_{k+1}) - \nabla f(\bm{x}_k))^\top (\bm{x}_{k+1} - \bm{x}_k) \ge \mu \alpha_k\norm{\bm{\eta}_k}^2$ in Euclidean space,
we need to show the Riemannian counterpart of this inequality, not that of~\eqref{eq:E_strong_convexity}.
Fortunately, by imposing the Wolfe conditions on the step length, we can directly prove the desired inequality~\eqref{eq:bound1} from the assumption of the strong convexity of $f$, i.e., the condition that the eigenvalue of the Riemannian Hessian is uniformly lower bounded.

The remainder of this paper is organized as follows.
Section \ref{sec:2} gives the mathematical preliminaries, including descriptions of retraction, vector transport, and existing RCG methods.
Section \ref{sec:3} presents our results for the HZ-type RCG method.
Section \ref{sec:4} provides numerical comparisons.
Section \ref{sec:5} briefly summarizes the key points.

\begin{landscape}
\begin{table}[htbp]
\caption{RCG results of previous studies and our results}
\label{table:1}
\begin{tabular}{lllll}
\bottomrule
\multirow{4}{*}{} & \multicolumn{4}{c}{Riemannian conjugate gradient methods} \\
\cmidrule(lr){2-5}                  
                  & \multicolumn{2}{c}{\multirow{2}{*}{Exponential retraction and}} & \multicolumn{2}{c}{\multirow{2}{*}{General retraction and}} \\
                  & & \\
                  & \multicolumn{2}{c}{its differentiation (as vector transport)}                  & \multicolumn{2}{c}{its (scaled) differentiation (as vector transport)} \\
\cmidrule(lr){2-5}                  
                  & Sufficient descent condition & Global convergence & Sufficient descent condition & Global convergence \\
\hline 
                
\multirow{2}{*}{} & \multirow{2}{*}{Ring--Wirth (2012) \cite{doi:10.1137/11082885X}} & \multirow{2}{*}{Ring--Wirth (2012) \cite{doi:10.1137/11082885X}} & \multirow{2}{*}{Ring--Wirth (2012) \cite{doi:10.1137/11082885X}} & \multirow{2}{*}{Ring--Wirth (2012) \cite{doi:10.1137/11082885X}} \\
\multirow{2}{*}{FR} & \multirow{2}{*}{Sato--Iwai (2015) \cite{sato2013new}} & \multirow{2}{*}{Sato--Iwai (2015) \cite{sato2013new}} & \multirow{2}{*}{Sato--Iwai (2015) \cite{sato2013new}} & \multirow{2}{*}{Sato--Iwai (2015) \cite{sato2013new}} \\
\multirow{2}{*}{} & \multirow{2}{*}{(strong Wolfe conditions)} & \multirow{2}{*}{(strong Wolfe conditions)} & \multirow{2}{*}{(strong Wolfe conditions)} & \multirow{2}{*}{(strong Wolfe conditions)} \\

\multirow{2}{*}{} & \multirow{2}{*}{} & \multirow{2}{*}{} & \multirow{2}{*}{} & \multirow{2}{*}{} \\

\hline
\multirow{2}{*}{DY} & \multirow{2}{*}{Sato (2016) \cite{sato2015dai}} & \multirow{2}{*}{Sato (2016) \cite{sato2015dai}} & \multirow{2}{*}{Sato (2016) \cite{sato2015dai}} & \multirow{2}{*}{Sato (2016) \cite{sato2015dai}} \\
\multirow{2}{*}{} & \multirow{2}{*}{(Wolfe conditions)} & \multirow{2}{*}{(Wolfe conditions)} & \multirow{2}{*}{(Wolfe conditions)} & \multirow{2}{*}{(Wolfe conditions)} \\
\multirow{2}{*}{} & \multirow{2}{*}{} & \multirow{2}{*}{} & \multirow{2}{*}{} & \multirow{2}{*}{} \\

\hline
\multirow{2}{*}{Hybrid} & \multirow{2}{*}{Sakai--Iiduka (2020) \cite{sakai2020hybrid}} & \multirow{2}{*}{Sakai--Iiduka (2020) \cite{sakai2020hybrid}} & \multirow{2}{*}{Sakai--Iiduka (2020) \cite{sakai2020hybrid}} & \multirow{2}{*}{Sakai--Iiduka (2020) \cite{sakai2020hybrid}} \\
\multirow{2}{*}{(HS--DY)} & \multirow{2}{*}{(strong Wolfe conditions)} & \multirow{2}{*}{(strong Wolfe conditions)} & \multirow{2}{*}{(strong Wolfe conditions)} & \multirow{2}{*}{(strong Wolfe conditions)} \\
\multirow{2}{*}{} & \multirow{2}{*}{} & \multirow{2}{*}{} & \multirow{2}{*}{} & \multirow{2}{*}{} \\

\hline
\multirow{2}{*}{Hybrid} & \multirow{2}{*}{Sakai--Iiduka (2021) \cite{sakai2020sufficient}} & \multirow{2}{*}{Sakai--Iiduka (2021) \cite{sakai2020sufficient}} & \multirow{2}{*}{Sakai--Iiduka (2021) \cite{sakai2020sufficient}} & \multirow{2}{*}{Sakai--Iiduka (2021) \cite{sakai2020sufficient}} \\
\multirow{2}{*}{(FR--PRP)} & \multirow{2}{*}{(strong Wolfe conditions)} & \multirow{2}{*}{(strong Wolfe conditions)} & \multirow{2}{*}{(strong Wolfe conditions)} & \multirow{2}{*}{(strong Wolfe conditions)} \\
\multirow{2}{*}{} & \multirow{2}{*}{} & \multirow{2}{*}{} & \multirow{2}{*}{} & \multirow{2}{*}{} \\

\hline
\multirow{2}{*}{} & \multirow{2}{*}{} & \multirow{2}{*}{} & \multirow{2}{*}{} & \multirow{2}{*}{} \\
\multirow{2}{*}{HZ} & \multirow{2}{*}{Sakai--Iiduka (2021) \cite{sakai2020sufficient}} & \multirow{2}{*}{\textbf{this work}} & \multirow{2}{*}{Sakai--Iiduka (2021) \cite{sakai2020sufficient}} & \multirow{2}{*}{------} \\
\multirow{2}{*}{} & \multirow{2}{*}{(without conditions)} & \multirow{2}{*}{(Wolfe conditions)} & \multirow{2}{*}{(without conditions)} & \multirow{2}{*}{} \\
\multirow{2}{*}{} & \multirow{2}{*}{} & \multirow{2}{*}{} & \multirow{2}{*}{} & \multirow{2}{*}{} \\
\toprule
\end{tabular}
See Section \ref{sec:2} for definitions of retraction, vector transport, FR, DY, HS-DY hybrid, and FR-PRP hybrid, and (strong) Wolfe conditions and Section \ref{sec:3} for definition of HZ-type RCG method.
\end{table} 
\end{landscape}

\section{Mathematical Preliminaries}\label{sec:2}
\subsection{Notation, definitions, and lemma}\label{subsec:2.1}
Let $(M, \ip{\cdot}{\cdot})$ be a connected geodesically complete Riemannian manifold, where $\ip{\cdot}{\cdot}_x:T_xM \times T_xM \rightarrow \real$ is a Riemannian metric at a point $x \in M$.
Here, $T_xM$ is a tangent space at a point $x \in M$, and $TM$ is a tangent bundle of $M$; i.e., $TM := \bigcup_{x \in M}T_xM$.
Let $\exp_x : T_xM \rightarrow M$ be the exponential map at $x \in M$ and $\oplus$ be the Whitney sum defined as follows (see \cite[Subchapter I.3 (p.16 (II))]{sakai1996riemannian}):
\begin{align*}
TM \oplus TM := \{(\xi , \eta) : \xi, \eta \in T_xM ,x \in M \}.
\end{align*}

An unconstrained optimization problem on $M$ is expressed as follows (see \cite{absil2008optimization, sakai2020hybrid, sakai2020sufficient, sato2015dai, sato2013new}):
\begin{problem}~\label{pbl:main}
Let $f:M\rightarrow\real$ be smooth. Then, we would like to
\begin{align*}
\text{minimize } f(x) \text{ subject to } x \in M.
\end{align*}
\end{problem}
To generalize line search optimization algorithms to Riemannian manifolds, the notions of a retraction and a vector transport are used.

\begin{definition}
[Retraction]~\label{def:retraction}
A retraction  (see \cite[Chapter 4, Definition 4.1.1]{absil2008optimization}) is a smooth map $R:TM{\rightarrow}M$ that has the following properties.
\begin{itemize}
\item $R_x(0_x)=x$;
\item With the canonical identification $T_{0_x}T_xM \simeq T_xM$, $R_x$ satisfies
\begin{align*}
(dR_x)_{0_x}(\xi)=\xi
\end{align*}
for all $\xi \in T_xM$,
\end{itemize}
where $0_x$ denotes the zero element of $T_xM$ and $R_x$ denotes the restriction of $R$ to $T_xM$.
\end{definition}

\begin{definition}
[Vector transport]~\label{def:vector transport}
A vector transport (see \cite[Chapter 8, Definition 8.1.1]{absil2008optimization}) is a smooth map $\tran:TM \oplus TM \rightarrow TM$ that has the following properties.
\begin{itemize}
\item There exists a retraction $R$, called the retraction associated with $\tran$, such that $\tran_{\eta}(\xi) \in T_{R_x(\eta)}M$ for all $x \in M$ and for all $\eta ,\xi \in T_xM$;
\item $\tran_{0_x}(\xi)=\xi$ for all $\xi \in T_xM$;
\item $\tran_{\eta}(a\xi + b\zeta)=a\tran_{\eta}(\xi)+b\tran_{\eta}(\zeta)$ for all $a,b \in \real$ and for all $\eta ,\xi ,\zeta \in T_xM$.
\end{itemize}
\end{definition}

\begin{lemma}[The Gauss lemma \cite{sakai1996riemannian}]~\label{lem:Gauss}
For any point $p \in M$, any $X \in T_pM$ and any $Y \in T_X(T_{p}M) \simeq T_pM$,
\begin{align*}
\ip{(d\exp_p)_X(X)}{(d\exp_p)_X(Y)}_{\exp_p(X)} = \ip{X}{Y}_p
\end{align*}
\end{lemma}

\subsection{Existing RCG Methods and Wolfe conditions}\label{subsec:2.2}
The RCG method \cite{absil2008optimization, sakai2020hybrid,sakai2020sufficient,sato2015dai,sato2013new} is described as
\begin{align}
x_{k+1} &= R_{x_k}(\alpha_k\eta_k), \label{eq:RCG1}\\
\eta_{k} &=
\begin{cases}
-g_k & k=0,\\
-g_k + \beta_k\tran_{\alpha_{k-1}\eta_{k-1}}(\eta_{k-1}) & k\geq 1,
\end{cases} \label{eq:RCG2}
\end{align}
where $g_k$ is the Riemannian gradient of $f$ at $x_k$, denoted by $\grad f(x_k)$,
$\alpha_k > 0$ is the positive step size, and $\beta_{k+1} \in \real$ is a parameter chosen suitably.
The $\beta_{k+1} \in \real$ parameters used in existing RCG methods are
\begin{align}
&\beta_{k+1}^\mathrm{FR}
= \frac{\|g_{k+1}\|_{x_{k+1}}^2}{\|g_{k}\|_{x_{k}}^2}, \label{FR} \\
&\beta_{k+1}^\mathrm{PRP}
= \frac{\ip{g_{k+1}}{y_{k}}_{x_{k+1}}}{\|g_{k}\|_{x_{k}}^2}, \\
&\beta_{k+1}^\mathrm{HS}
=\frac{\ip{g_{k+1}}{y_{k}}_{x_{k+1}}}{\ip{g_{k+1}}{\tran_{\alpha_k\eta_k}(\eta_k)}_{x_{k+1}}-\ip{g_{k}}{\eta_{k}}_{x_{k}}}, \label{eq:oldHS} \\
&\beta_{k+1}^\mathrm{DY}
= \frac{\|g_{k+1}\|_{x_{k+1}}^2}{\ip{g_{k+1}}{\tran_{\alpha_k\eta_k}(\eta_k)}_{x_{k+1}}-\ip{g_{k}}{\eta_{k}}_{x_{k}}}, \label{DY} \\
&\beta_{k+1}^\mathrm{Hybrid(HS-DY)}
= \max \left\{ 0, \min \left\{ \beta_{k+1}^\mathrm{HS}, \beta_{k+1}^\mathrm{DY} \right\} \right\}, \label{HS-DY} \\
&\beta_{k+1}^\mathrm{Hybrid(FR-PRP)}
= \max \left\{ 0, \min \left\{ \beta_{k+1}^\mathrm{FR}, \beta_{k+1}^\mathrm{PRP} \right\} \right\}, \label{FR-PRP}
\end{align}
where $y_{k} := g_{k+1} - \tran_{\alpha_k\eta_k}(g_k)$.
To determine step size $\alpha_k$ in \eqref{eq:RCG1}, we use line searches that satisfy the {\em Wolfe conditions} (see \cite{sakai2020hybrid, sakai2020sufficient, sato2015dai, sato2013new}),
\begin{align}
&f(R_{x_k}(\alpha_k\eta_k)) \leq f(x_k)+c_1\alpha_k\ip{g_k}{\eta_k}_{x_k}, \label{eq:Wolfe1}\\
&\ip{g_{k+1}}{\tran_{\alpha_k\eta_k}(\eta_k)}_{x_{k+1}} \geq c_2 \ip{g_k}{\eta_k}_{x_k}, \label{eq:Wolfe2}
\end{align}
where $0 < c_1 < c_2 <1$.
When \eqref{eq:Wolfe2} is replaced with 
\begin{align}
\left| \ip{g_{k+1}}{\tran_{\alpha_k\eta_k}(\eta_k)}_{x_{k+1}} \right| 
\leq c_2 \left| \ip{g_k}{\eta_k}_{x_k} \right|, \label{eq:Wolfe3}
\end{align}
\eqref{eq:Wolfe1} and \eqref{eq:Wolfe3} are called {\em strong Wolfe conditions}. 

Search direction $\eta_k$ defined by \eqref{eq:RCG2} is said to be 
a {\em sufficient descent direction} if there exists $\kappa > 0$ such that, for all $k = 0,1,\ldots$, 
\begin{align*}
\langle g_k, \eta_k \rangle \leq - \kappa \|g_k\|_{x_k}^2.
\end{align*}

Let us first consider the FR-type RCG method, i.e., the RCG method \eqref{eq:RCG1} and \eqref{eq:RCG2}, using a general retraction and scaled vector transport,
with \eqref{FR}. It is guaranteed to generate a sufficient descent direction and to converge globally under strong Wolfe conditions [\eqref{eq:Wolfe1} and \eqref{eq:Wolfe3}],
\cite{doi:10.1137/11082885X,sato2013new} (see also Table \ref{table:1})
 
Next, let us consider the DY-type RCG method, i.e., the RCG method \eqref{eq:RCG1} and \eqref{eq:RCG2}, using a general retraction and scaled vector transport, with \eqref{DY}.
It is guaranteed to generate a sufficient descent direction and to converge globally under Wolfe conditions [\eqref{eq:Wolfe1} and \eqref{eq:Wolfe2}] \cite{sato2015dai} (see also Table \ref{table:1}).
A hybrid method using either \eqref{HS-DY} or \eqref{FR-PRP} also generates a sufficient descent direction and converges globally \cite{sakai2020hybrid,sakai2020sufficient} (see also Table \ref{table:1}).

\section{HZ-type RCG Method}\label{sec:3}
\subsection{Assumptions}
The parameter $\beta_{k+1}$ used in the HZ-type RCG method \cite{sakai2020hybrid, sakai2020sufficient} is defined by
\begin{align}~\label{eq:oldHZ}
\beta_{k+1}^\mathrm{HZ} = \beta^\mathrm{HS}_{k+1} -
\mu\frac{\norm{y_{k}}_{x_{k+1}}^2\ip{g_{k+1}}{\tran_{\alpha_k\eta_k}(\eta_k)}_{x_{k+1}}}{\left(\ip{g_{k+1}}{\tran_{\alpha_k\eta_k}(\eta_k)}_{x_{k+1}}-\ip{g_{k}}{\eta_{k}}_{x_{k}}\right)^2},
\end{align}
where $\mu > 1/4$ and $y_{k} := g_{k+1} - \tran_{\alpha_k\eta_k}(g_k)$.

In this paper, we use the exponential map as a retraction, i.e., $R := \exp$.
Moreover, we use the vector transport defined by the differential of the exponential retraction; i.e.,
\begin{align*}
\tran : TM \oplus TM \rightarrow TM : (\eta ,\xi) \mapsto \tran_{\eta}(\xi) := (d\exp_x)_\eta(\xi),
\end{align*}
for $\eta ,\xi \in T_xM$.
From the Gauss lemma (Lemma \ref{lem:Gauss}), we have
\begin{align}~\label{eq:preserve}
\ip{g_{k}}{\eta_{k}}_{x_{k}} = \ip{\tran_{\alpha_k\eta_k}(g_{k})}{\tran_{\alpha_k\eta_k}(\eta_{k})}_{x_{k+1}}.
\end{align}
This means that \eqref{eq:oldHZ} and \eqref{eq:oldHS} can be written as
\begin{align}
\beta_{k+1}^\mathrm{HZ} &= \beta^\mathrm{HS}_{k+1} -
\mu\frac{\norm{y_{k}}_{x_{k+1}}^2\ip{g_{k+1}}{\tran_{\alpha_k\eta_k}(\eta_k)}_{x_{k+1}}}{\ip{y_{k}}{\tran_{\alpha_k\eta_k}(\eta_k)}_{x_{k+1}}^2}, \label{eq:HZ} \\
\beta_{k+1}^\mathrm{HS}&=\frac{\ip{g_{k+1}}{y_{k}}_{x_{k+1}}}{\ip{y_{k}}{\tran_{\alpha_k\eta_k}(\eta_k)}_{x_{k+1}}}, \label{eq:HS}
\end{align}
respectively.
Therefore, the HZ-type RCG method with the exponential retraction can be described as Algorithm \ref{alg:RCG}.
\begin{algorithm}
\caption{HZ-type RCG method with exponential retraction for solving Problem \ref{pbl:main} \cite{absil2008optimization, sakai2020hybrid, sakai2020sufficient} \label{alg:RCG}}
\begin{algorithmic}[1]
\REQUIRE Initial point $x_0 \in M$, convergence tolerance $\epsilon > 0$.
\ENSURE Sequence $\{x_k\}_{k=0,1,\cdots} \subset M$.
\STATE Set $\eta_0 = -g_0 := -\grad f(x_0)$.
\STATE $k \leftarrow 0.$
\WHILE{$\norm{g_k}_{x_k} > \epsilon$}
\STATE Compute $\alpha_k > 0$ satisfying Wolfe conditions \eqref{eq:Wolfe1} and \eqref{eq:Wolfe2}.
\STATE Set \begin{align*}x_{k+1}=\exp_{x_k}(\alpha_k\eta_k),\end{align*}
\STATE Compute $g_{k+1} := -\grad f(x_{k+1})$ and $\beta_{k+1}$ as \eqref{eq:HZ} and set search direction 
\begin{align*}\eta_{k+1} = -g_{k+1} + \beta_{k+1}(d\exp_{x_k})_{\alpha_k\eta_k}(\eta_k). \end{align*}
\STATE $k \leftarrow k + 1.$
\ENDWHILE
\end{algorithmic}
\end{algorithm}
In addition, we also consider the modified HZ method (see \cite[(1.6)]{hager2005new}), by replacing $\beta_{k+1}^\mathrm{HZ}$ in step 6 of Algorithm \ref{alg:RCG} by
\begin{align}~\label{eq:modified HZ}
\hat{\beta}_{k+1}^\mathrm{HZ} := \max\{\beta_{k+1}^\mathrm{HZ},\zeta_{k+1}\}, \quad \zeta_{k+1} := -\frac{1}{\norm{\eta_{k+1}}_{x_{k+1}}\min\{\zeta ,\norm{g_{k+1}}_{x_{k+1}}\}},
\end{align}
where $\zeta > 0$ is a constant.

We also consider the modified HZ method (see \cite[(1.6)]{hager2005new}) by replacing $\beta_{k+1}^\mathrm{HZ}$ in step 6 of Algorithm \ref{alg:RCG} with
\begin{align}~\label{eq:modified HZ}
\hat{\beta}_{k+1}^\mathrm{HZ} := \max\{\beta_{k+1}^\mathrm{HZ},\zeta_{k+1}\}, \quad \zeta_{k+1} := -\frac{1}{\norm{\eta_{k+1}}_{x_{k+1}}\min\{\zeta ,\norm{g_{k+1}}_{x_{k+1}}\}},
\end{align}
where $\zeta > 0$ is a constant.

We consider Algorithm \ref{alg:RCG} under Assumption \ref{asm:Lipschitz} (see \cite[Theorem 2]{doi:10.1137/11082885X}) and Assumption \ref{asm:A} described below.

\begin{assumption}~\label{asm:Lipschitz}
The objective function $f:M\rightarrow\real$ in Problem \ref{pbl:main} is smooth and bounded below, and $f \circ \exp_{x_k}:T_{x_k}M \rightarrow \real$ is Lipschitz continuously differentiable on $\mathrm{span}\{\eta_k\}$ with uniform Lipschitz constant $L>0$.
\end{assumption}

The following is Zoutendijk's theorem (Theorem \ref{thm:Zoutendijk}) for Riemannian manifolds under Assumption \ref{asm:Lipschitz}.

\begin{theorem}[Zoutendijk]~\label{thm:Zoutendijk}
Let $\{x_k\}_{k=0,1,\cdots} \subset M$ be a sequence generated by Algorithm \ref{alg:RCG}.
Suppose that Assumption \ref{asm:Lipschitz} holds.
If each step size $\alpha_k>0$ satisfies Wolfe conditions \eqref{eq:Wolfe1} and \eqref{eq:Wolfe2}, then
\begin{align}
\label{eq:Zoutendijk}
\sum_{k=0}^\infty \frac{\ip{g_k}{\eta_k}_{x_k}^2}{\norm{\eta_k}^2_{x_k}} < \infty .
\end{align}
\end{theorem}

\begin{assumption}~\label{asm:A}
The objective function $f:M\rightarrow\real$ in Problem \ref{pbl:main} is smooth, and
there exists a constant $L > 0$ such that, for all $x,y \in M$,
\begin{align}
\norm{\grad f(x) - \tran_{X}(\grad f(y))}_x \leq Ld(x,y), \label{eq:Lipschitz}
\end{align}
where $X \in T_yM$ satisfies $x = \exp_y(X)$.
Furthermore, $f$ is strongly convex, i.e., there exists a constant $\mu > 0$ such that, for all $x \in M$, the smallest eigenvalue of the Riemannian Hessian $\Hess f(x)$ is not less than $\mu$.
\end{assumption}


\subsection{Convergence results}
Our first result is that Algorithm \ref{alg:RCG} including the HZ-type RCG method generates a sufficient descent direction without depending on the line search conditions.

\begin{theorem}~\label{thm:SDP}
Let $\{x_k\}_{k=0,1,\cdots} \subset M$ be a sequence generated by Algorithm \ref{alg:RCG} with $\beta_k \in [\beta_k^\mathrm{HZ}, \max\{\beta_k^\mathrm{HZ},0\}]$
\footnote{The modified HZ \eqref{eq:modified HZ} satisfies $\hat{\beta}_k^\mathrm{HZ} \in [\beta_k^\mathrm{HZ}, \max\{\beta_k^\mathrm{HZ},0\}]$.}.
If $\ip{y_k}{\tran_{\alpha_k\eta_k}(\eta_k)}_{x_{k+1}}\neq 0$, we have
\begin{align}~\label{eq:SDP}
\ip{g_{k}}{\eta_{k}}_{x_{k}} \leq -\left(1-\frac{1}{4\mu}\right)\norm{g_{k}}_{x_{k}}^2.
\end{align}
\end{theorem}
\begin{proof}
For $k = 0$, \eqref{eq:SDP} clearly holds from $\ip{g_0}{\eta_0}_{x_{0}} = -\norm{g_0}_{x_0}^2$.
Subsequently, we assume $k \ge 1$.
If $\beta_k = \beta_k^\mathrm{HZ}$, from \cite[Theorem 3.4]{sakai2020sufficient}, \eqref{eq:SDP} follows.
On the other hand, if $\beta_k\neq\beta_k^\mathrm{HZ}$, then $\beta_k^\mathrm{HZ} \leq \beta_k \leq 0$.
From \eqref{eq:RCG2}, we have
\begin{align*}
\ip{g_{k}}{\eta_{k}}_{x_{k}} = -\norm{g_{k}}_{x_{k}}^2 + \beta_{k}\ip{g_{k}}{\tran_{\alpha_{k-1}\eta_{k-1}}(\eta_{k-1})}_{x_{k}}.
\end{align*}
If $\ip{g_{k}}{\tran_{\alpha_{k-1}\eta_{k-1}}(\eta_{k-1})}_{x_{k}} \geq 0$, then \eqref{eq:SDP} follows immediately since $\beta_{k} \leq 0$.
If $\ip{g_{k}}{\tran_{\alpha_{k-1}\eta_{k-1}}(\eta_{k-1})}_{x_{k}} < 0$, then
\begin{align*}
\ip{g_{k}}{\eta_{k}}_{x_{k}} &= -\norm{g_{k}}_{x_{k}}^2 + \beta_{k}\ip{g_{k}}{\tran_{\alpha_{k-1}\eta_{k-1}}(\eta_{k-1})}_{x_{k}} \\
&\leq -\norm{g_{k}}_{x_{k}}^2 + \beta^\mathrm{HZ}_{k}\ip{g_{k}}{\tran_{\alpha_{k-1}\eta_{k-1}}(\eta_{k-1})}_{x_{k}},
\end{align*}
since $\beta_k^\mathrm{HZ} \leq \beta_k \leq 0$.
Hence, \eqref{eq:SDP} follows by analysis as in \cite[Theorem 3.4]{sakai2020sufficient}.
\end{proof}

The following is the main theorem indicating that the HZ-type RCG method converges globally.

\begin{theorem}~\label{thm:converge1}
Let $\{x_k\}_{k=0,1,\cdots} \subset M$ be a sequence generated by Algorithm \ref{alg:RCG} with $\beta_{k+1} = \beta_{k+1}^\mathrm{HZ}$
under Assumptions \ref{asm:Lipschitz} and \ref{asm:A}.
Suppose that each step size $\alpha_k>0$ satisfies Wolfe conditions \eqref{eq:Wolfe1} and \eqref{eq:Wolfe2}. 
Then either $\norm{g_{k_0}}_{x_{k_0}}=0$ for some $k_0 \in \nat$, or
\begin{align}~\label{eq:converge1}
\lim_{k \to \infty}\norm{g_k}_{x_k} = 0.
\end{align}
\end{theorem}

\begin{proof}
If $g_{k_0} = 0$ for some $k_0 \in \nat$, then \eqref{eq:converge1} obviously follows.
Assume that $g_k \neq 0$ for all $k \in \nat$.
From \eqref{eq:preserve} and \eqref{eq:Wolfe2}, we have
\begin{align*}
(c_2-1) \ip{g_k}{\eta_k}_{x_k} \leq \ip{g_{k+1}-\tran_{\alpha_k\eta_k}(g_{k})}{\tran_{\alpha_k\eta_k}(\eta_k)}_{x_{k+1}}.
\end{align*}
Moreover, \eqref{eq:Lipschitz} and the Cauchy--Schwarz inequality imply
\begin{align*}
\ip{g_{k+1}-\tran_{\alpha_k\eta_k}(g_{k})}{\tran_{\alpha_k\eta_k}(\eta_k)}_{x_{k+1}} &\leq \norm{g_{k+1}-\tran_{\alpha_k\eta_k}(g_{k})}_{x_{k+1}}\norm{\tran_{\alpha_k\eta_k}(\eta_k)}_{x_{k+1}} \\
&=\norm{g_{k+1}-\tran_{\alpha_k\eta_k}(g_{k})}_{x_{k+1}}\norm{\eta_k}_{x_{k}} \\
&\leq L\alpha_k\norm{\eta_k}_{x_{k}}^2.
\end{align*}
Therefore, we obtain
\begin{align*}
(c_2-1) \ip{g_k}{\eta_k}_{x_k} \leq L\alpha_k\norm{\eta_k}_{x_{k}}^2,
\end{align*}
which with Theorem \ref{thm:SDP} implies
\begin{align}~\label{eq:lower_step}
\alpha_k \geq \frac{1-c_2}{L}\frac{\abs{\ip{g_k}{\eta_k}_{x_k}}}{\norm{\eta_k}_{x_{k}}^2}.
\end{align}
Moreover,
\begin{align}
\ip{y_k}{\tran_{\alpha_k\eta_k}(\eta_k)}_{x_{k+1}} &= \ip{g_{k+1}-\tran_{\alpha_k\eta_k}(g_k)}{\tran_{\alpha_k\eta_k}(\eta_k)}_{x_{k+1}} \nonumber \\
&= \ip{g_{k+1}}{\tran_{\alpha_k\eta_k}(\eta_k)}_{x_{k+1}} - \ip{\tran_{\alpha_k\eta_k}(g_k)}{\tran_{\alpha_k\eta_k}(\eta_k)}_{x_{k+1}} \nonumber \\
&= \ip{g_{k+1}}{\tran_{\alpha_k\eta_k}(\eta_k)}_{x_{k+1}} - \ip{g_k}{\eta_k}_{x_k} \nonumber \\
&\geq c_2\ip{g_k}{\eta_k}_{x_k} - \ip{g_k}{\eta_k}_{x_k} \nonumber \\
&= (c_2-1)\ip{g_k}{\eta_k}_{x_k}, \label{eq:temp1}
\end{align}
where the third equation comes from \eqref{eq:preserve}, and the inequality comes from \eqref{eq:Wolfe2}.
Furthermore, we define $\phi_{x,\mu}(t) := f(\exp_x(t\mu))$.
From Taylor's theorem, we have
\begin{align*}
f(x_{k+1}) - f(x_k) &= f(\exp_{x_k}(\alpha_k\eta_k)) - f(x_k) \\
&= \phi_{x_k,\frac{\eta_k}{\norm{\eta_k}_{x_k}}}(\alpha_k\norm{\eta_k}_{x_k}) - \phi_{x_k,\frac{\mu_k}{\norm{\eta_k}_{x_k}}}(0) \\
&= \phi^\prime_{x_k,\frac{\eta_k}{\norm{\eta_k}_{x_k}}}(0)\alpha_k\norm{\eta_k}_{x_k} +\frac{1}{2} \phi^{\prime\prime}_{x_k,\frac{\eta_k}{\norm{\eta_k}_{x_k}}}(\theta)(\alpha_k\norm{\eta_k}_{x_k})^2,
\end{align*}
for some $\theta \in [0, \alpha_k\norm{\eta_k}_{x_k}]$.
Moreover, by defining $c_{x,\mu}(t) := \exp_{x}(t\mu)$, we have
\begin{align*}
\phi''_{x_k,\frac{\eta_k}{\norm{\eta_k}_{x_k}}}(\theta)
& = \ip{\Hess f(x_k)\left[c'_{x_k,\frac{\eta_k}{\norm{\eta_k}_{x_k}}}(\theta)\right]}{c'_{x_k,\frac{\eta_k}{\norm{\eta_k}_{x_k}}}(\theta)}_{c_{x_k,\frac{\eta_k}{\norm{\eta_k}_{x_k}}}(\theta)}\\
& \ge \mu\norm{c'_{x_k,\frac{\eta_k}{\norm{\eta_k}_{x_k}}}(\theta)}_{c_{x_k,\frac{\eta_k}{\norm{\eta_k}_{x_k}}}(\theta)}^2\\
& = \mu\norm{\frac{\eta_k}{\norm{\eta_k}_{x_k}}}_{x_k}^2\\
& = \mu
\end{align*}
from the strong convexity of $f$.
Using this evaluation of $\phi''$, we obtain
\begin{align*}
f(x_{k+1}) - f(x_k) & \geq \ip{g_k}{\frac{\eta_k}{\norm{\eta_k}_{x_k}}}_{x_k} \alpha_k\norm{\eta_k}_{x_k} + \frac{\mu}{2}(\alpha_k\norm{\eta_k}_{x_k})^2 \\
&=  \alpha_k\ip{g_k}{\eta_k}_{x_k} + \frac{\mu}{2}(\alpha_k\norm{\eta_k}_{x_k})^2 
\end{align*}
Therefore, we obtain
\begin{align}~\label{eq:diff1}
f(x_{k+1}) - f(x_k) \geq \alpha_k\ip{g_k}{\eta_k}_{x_k} + \frac{\mu}{2}(\alpha_k\norm{\eta_k}_{x_k})^2.
\end{align}
From \eqref{eq:Wolfe1} and \eqref{eq:diff1}, we have
\begin{align}~\label{eq:temp2}
\ip{g_k}{\eta_k}_{x_k} \leq \frac{\mu}{2(c_1-1)}\alpha_k\norm{\eta_k}_{x_k}^2.
\end{align}
Therefore, from \eqref{eq:temp1} and \eqref{eq:temp2}, we obtain
\begin{align}~\label{eq:bound1}
\ip{y_k}{\tran_{\alpha_k\eta_k}(\eta_k)}_{x_{k+1}} \geq \gamma \alpha_k \norm{\eta_k}_{x_k}^2,
\end{align}
where
\begin{align*}
\gamma := \frac{\mu(1-c_2)}{2(1-c_1)}.
\end{align*}
The assumption $g_k \neq 0$ implies that $\eta_k \neq 0$,
which together with $\alpha_k > 0$ yield $\ip{y_k}{\tran_{\alpha_k\eta_k}(\eta_k)}_{x_{k+1}} \neq 0$.
From \eqref{eq:Wolfe1} and the lower boundedness of $f$,
\begin{align*}
\sum_{k=0}^\infty c_1\alpha_k\ip{g_k}{\eta_k}_{x_k} &\geq \sum_{k=0}^\infty (f(x_{k+1}) - f(x_k)) \\
& = \lim_{j \to +\infty}f(x_j) - f(x_0) > -\infty ,
\end{align*}
which implies
\begin{align*}
\sum_{k=0}^\infty \alpha_k\ip{g_k}{\eta_k}_{x_k} > -\infty .
\end{align*}
Combining this with the lower bound for $\alpha_k$ given in \eqref{eq:lower_step} and the sufficient descent property in Theorem \ref{thm:SDP} gives
\begin{align}~\label{eq:sum_bound}
\sum_{k=0}^\infty\frac{\norm{g_k}_{x_k}^4}{\norm{\eta_k}_{x_k}^2} < \infty .
\end{align}
From \eqref{eq:Lipschitz}, we obtain
\begin{align}~\label{eq:bound2}
\norm{y_k}_{x_{k+1}} = \norm{g_{k+1} - \tran_{\alpha_k\eta_k}(g_k)}_{x_{k+1}} \leq L\alpha_k\norm{\eta_k}_{x_k}.
\end{align}
From \eqref{eq:HZ}, \eqref{eq:HS}, \eqref{eq:bound1} and \eqref{eq:bound2}, we have
\begin{align*}
\abs{\beta_{k+1}^\mathrm{HZ}} &= \abs{\frac{\ip{g_{k+1}}{y_{k}}_{x_{k+1}}}{\ip{y_{k}}{\tran_{\alpha_k\eta_k}(\eta_k)}_{x_{k+1}}} - \mu\frac{\norm{y_{k}}_{x_{k+1}}^2\ip{g_{k+1}}{\tran_{\alpha_k\eta_k}(\eta_k)}_{x_{k+1}}}{\ip{y_{k}}{\tran_{\alpha_k\eta_k}(\eta_k)}_{x_{k+1}}^2}} \\
&\leq \frac{\norm{g_{k+1}}_{x_{k+1}}\norm{y_{k}}_{x_{k+1}}}{\gamma \alpha_k \norm{\eta_k}_{x_k}^2}
+ \mu\frac{\norm{y_{k}}_{x_{k+1}}^2\norm{g_{k+1}}_{x_{k+1}}\norm{\eta_k}_{x_{k}}}{\gamma^2 \alpha_k^2 \norm{\eta_k}_{x_k}^4} \\
& \leq \frac{L\alpha_k\norm{\eta_k}_{x_k}\norm{g_{k+1}}_{x_{k+1}}}{\gamma \alpha_k \norm{\eta_k}_{x_k}^2}
+ \mu\frac{L^2\alpha_k^2\norm{\eta_k}_{x_k}^3\norm{g_{k+1}}_{x_{k+1}}}{\gamma^2 \alpha_k^2 \norm{\eta_k}_{x_k}^4} \\
&= \left( \frac{L}{\gamma}+\frac{\mu L^2}{\gamma^2}\right)\frac{\norm{g_{k+1}}_{x_{k+1}}}{\norm{\eta_k}_{x_k}}.
\end{align*}
Hence, we have
\begin{align*}
\norm{\eta_{k+1}}_{x_{k+1}} &= \norm{-g_{k+1} + \beta_{k+1}^\mathrm{HZ}\tran_{\alpha_{k}\eta_{k}}(\eta_{k})}_{x_{k+1}} \\
& \leq \norm{g_{k+1}} + \abs{\beta_{k+1}^\mathrm{HZ}} \norm{\eta_{k}}_{x_{k}} \\
& \leq \left(1 + \frac{L}{\gamma}+\frac{\mu L^2}{\gamma^2}\right)\norm{g_{k+1}}_{x_{k+1}}.
\end{align*}
Combining this upper bound with \eqref{eq:sum_bound}, we obtain
\begin{align*}
\sum_{k=0}^\infty \norm{g_k}_{x_k}^2 < \infty ,
\end{align*}
which completes the proof.
\end{proof}

\subsection{Comparison of HZ-type ECG method with HZ-type RCG method}
Let us consider $M = \mathbb{R}^n$. 
Then, $\beta_{k+1}^\mathrm{HZ}$ defined by \eqref{eq:HZ} can be expressed
\begin{align*}
\beta_{k+1}^\mathrm{HZ} 
&= \beta^\mathrm{HS}_{k+1} -
\mu\frac{\norm{y_{k}}_{x_{k+1}}^2\ip{g_{k+1}}{\tran_{\alpha_k\eta_k}(\eta_k)}_{x_{k+1}}}{\ip{y_{k}}{\tran_{\alpha_k\eta_k}(\eta_k)}_{x_{k+1}}^2}\\
&= 
\frac{\bm{y}_{k}^\top \bm{g}_{k+1}}{\bm{\eta}_k^\top \bm{y}_{k}}
- \mu \frac{\| \bm{y}_{k} \|^2 \bm{g}_{k+1}^\top \bm{\eta}_k}{(\bm{y}_{k}^\top \bm{\eta}_{k})^2}\\
&=
\frac{1}{\bm{\eta}_k^\top \bm{y}_{k}} 
\left( \bm{y}_{k}  - \mu \frac{\| \bm{y}_{k} \|^2}{\bm{\eta}_k^\top \bm{y}_{k}} \bm{\eta}_k  \right)^\top \bm{g}_{k+1},
\end{align*}
which implies that $\beta_{k+1}^\mathrm{HZ}$ defined by \eqref{eq:HZ} with $\mu = 2$ coincides with (1.3) in \cite{hager2005new} used in the HZ-type ECG method.
Inequality \eqref{eq:SDP} with $\mu = 2$ (see Theorem \ref{thm:SDP}) is the sufficient descent property of the HZ-type RCG method; i.e.,
\begin{align*}
\ip{g_{k}}{\eta_{k}}_{x_{k}} \leq - \frac{7}{8} \norm{g_{k}}_{x_{k}}^2,
\end{align*}
which, together with $M = \mathbb{R}^n$, implies (1.9) in Theorem 1.1 of \cite{hager2005new}:
\begin{align*}
\bm{g}_{k}^\top \bm{\eta}_{k} \leq - \frac{7}{8} \norm{\bm{g}_{k}}^2.
\end{align*}
Accordingly, Theorem \ref{thm:SDP} is a natural extended result of Theorem 1.1 in \cite{hager2005new} to a Riemannian manifold.

Theorem 2.2 in \cite{hager2005new} implies that the HZ-type ECG method converges globally if the Wolfe conditions hold and if 
\begin{itemize}
\item $f \colon \mathbb{R}^n \to \mathbb{R}$ is strongly convex with a constant $c > 0$ and $\nabla f \colon \mathbb{R}^n \to \mathbb{R}^n$
is Lipschitz continuous with Lipschitz constant $L > 0$ on the level set $\mathcal{L} := \{\bm{x} \in \mathbb{R}^n \colon f(\bm{x}_0) \leq f(\bm{x})\}$.
\end{itemize}
Theorem \ref{thm:converge1} in this paper is satisfied under the Wolfe conditions and Assumptions \ref{asm:Lipschitz} and \ref{asm:A}, i.e.,
\begin{enumerate}
\item[(i)] $f \colon M\to \mathbb{R}$ is smooth and bounded below; $f \circ \exp_{x_k} \colon T_{x_k}M \to \real$ is Lipschitz continuously differentiable on $\mathrm{span}\{\eta_k\}$ with uniform Lipschitz constant $L>0$;
\item[(ii)] There exists a constant $L > 0$ such that, for all $x,y \in M$,
\begin{align}\label{eq:Lipschitz1}
\norm{\grad f(x) - \tran_{X}(\grad f(y))}_x \leq Ld(x,y), 
\end{align}
where $X \in T_yM$ satisfies $x = \exp_y(X)$. Furthermore, $f$ is strongly convex; i.e., there exists constant $\mu > 0$ such that, for all $x \in M$, the smallest eigenvalue of the Riemannian Hessian $\Hess f(x)$ is not less than $\mu$.
\end{enumerate}

Under the Euclidean space setting, $f$ is strongly convex with a constant $c$ if and only if the smallest eigenvalue of $\nabla^2 f (\bm{x})$ for any $\bm{x} \in \mathbb{R}^n$ is not less than $c$.
Moreover, \eqref{eq:Lipschitz1} in the Euclidean space setting is the same as the existence of $L>0$ such that, for all $\bm{x},\bm{y}\in \mathbb{R}^n$,
\begin{align}\label{L}
\| \nabla f (\bm{x}) - \nabla f (\bm{y})\| \leq L \|\bm{x}-\bm{y}\|,
\end{align}
that is, $\nabla f$ is Lipschitz continuous. Obviously, the strong convexity of $f$ implies that $f$ is bounded below.
Following Assumption 4.1 and Remark 4.1 in \cite{sato2021riemannian}, we can see that, in the Euclidean space setting,
the Lipschitz continuity of $f \circ \exp_{x_k}$ (see (i)) is equivalent to \eqref{L}. Therefore, Theorem \ref{thm:converge1} is a natural extended result of Theorem 2.2 in [3] to a Riemannian manifold.

%

\section{Numerical Experiments}\label{sec:4}
We compared the performances of the HZ method with existing RCG methods, i.e., the FR, DY, PRP, HS, HS--DY hybrid, and FR--PRP hybrid methods.
We solved two Riemannian optimization problems (Problem \ref{pbl:rayleigh} and \ref{pbl:stability}) on the unit sphere on a MacBook Air laptop computer (2020) with a 1.1-GHz Intel Core i3 CPU,
8-GB 3733-MHz LPDDR4X memory, and the Catalina 10.15.7 OS. The algorithms were written in Python 3.9.12. Each problem was solved 100 times with each algorithm, that is, 200 times in total.

We used a line search algorithm \cite[Algorithm 3]{sakai2020sufficient} for the strong Wolfe conditions \eqref{eq:Wolfe1} and \eqref{eq:Wolfe3} with $c_1 = 10^{-4}$ and $c_2 = 0.9$. If 
\begin{align*}
\norm{\grad f(x_k)}_{x_k} < 10^{-6},
\end{align*}
was satisfied, we determined that the sequence had converged to an optimal solution.

For comparison, we calculated performance profile $P_s:\real\rightarrow[0,1]$ \cite{dolan2002benchmarking},
defined as follows. Let $\mathcal{P}$ and $\mathcal{S}$ be the set of problems and solvers, respectively. For each $p \in \mathcal{P}$ and $s \in \mathcal{S}$, we define
\begin{align*}
t_{p,s} := (\text{iterations or time required to solve problem }p\text{ by solver }s).
\end{align*}
We define performance ratio $r_{p,s}$ as
\begin{align*}
r_{p,s} := \frac{t_{p,s}}{\min_{s^\prime \in \mathcal{S}}t_{p,s^\prime}}
\end{align*}
and define the performance profile for all $\tau \in \real$ as
\begin{align*}
P_{s}(\tau) := \frac{\abs{\{p \in \mathcal{P} : r_{p,s} \leq \tau\}}}{\abs{\mathcal{P}}},
\end{align*}
where $\abs{A}$ denotes the number of elements in set $A$.

\subsection{The Rayleigh quotient minimization problem on the unit sphere}
Problem \ref{pbl:rayleigh} is the Rayleigh-quotient minimization problem on the unit sphere.
The optimal solutions are the unit eigenvectors of $A$ associated with the smallest eigenvalue (see \cite[Chapter 4.6]{absil2008optimization}).

\begin{problem}\label{pbl:rayleigh}
For a symmetric positive-definite matrix $A$,
\begin{align*}
\text{minimize}\quad& f(x) := x^\top Ax,\\
\text{subject to}\quad& x \in \mathbb{S}^{n - 1} := \{ x \in \real^n : \norm{x} = 1 \},
\end{align*}
where $\norm{\cdot}$ denotes the Euclidean norm.
\end{problem}

In the experiments, we generated a matrix $A$ randomly with $n = 100$ by using \texttt{sklearn.datasets.make\_spd\_matrix}.

Figure \ref{fig1} plots the performance profile of each algorithm versus the number of iterations.
It shows that the HZ method had much better performance than the FR, DY, PRP, and HS methods.
Figure \ref{fig2} plots the performance profile of each algorithm versus the elapsed time.
It also shows that the performance of the HZ method was much better than those of the FR, DY, PRP, and HS methods.
The two hybrid methods had even better performance. In particular, the figures show that they are suitable for solving Problem \ref{pbl:rayleigh}.

\begin{figure}[htbp]
\centering
 \includegraphics[scale=0.5]{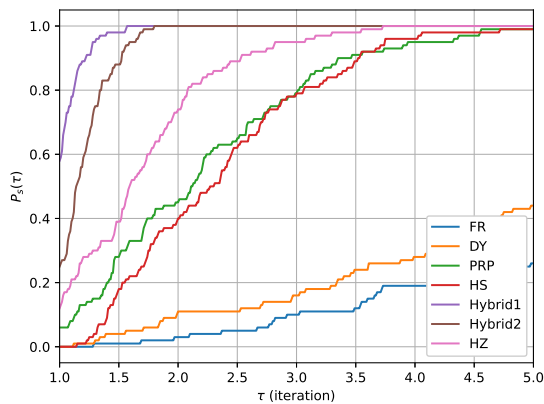}
\caption{Performance profile versus number of iterations for Problem \ref{pbl:rayleigh}. \label{fig1}}
\end{figure}

\begin{figure}[htbp]
\centering
 \includegraphics[scale=0.5]{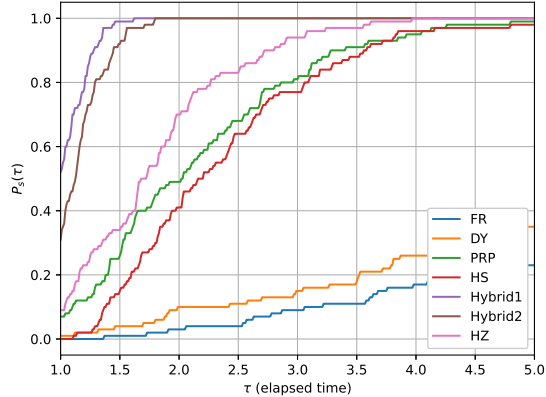}
\caption{Performance profile versus elapsed time for Problem \ref{pbl:rayleigh}. \label{fig2}}
\end{figure}

\subsection{Computation of stability number}
We define the stability number $S(G)$ of an undirected graph $G = (E, V)$ as the size of the maximum stable set in $G$.
Motzkin and Straus showed that solving the stability number of graphs problem is equivalent to solving Problem \ref{pbl:stability} \cite{motzkin1965maxima}.

\begin{problem}\label{pbl:stability}
For an undirected graph $G = (E,V)$,
\begin{align*}
\text{minimize}\quad& f(x) := \sum_{i=1}^n x_i^4 + \sum_{(i,j) \in E} x_i^2x_j^2, \\
\text{subject to}\quad& x \in \mathbb{S}^{n - 1} := \{ x \in \real^n : \norm{x} = 1 \},
\end{align*}
where $n := \abs{V}$ and $\norm{\cdot}$ denotes the Euclidean norm.
\end{problem}

In the experiments, we generated a graph $G = (E, V)$ randomly with $n = 100$ by using \texttt{networkx.fast\_gnp\_random\_graph}.
We set the probability for edge creation to 0.1.

Figure~\ref{fig3} plots the performance profile of each algorithm versus the number of iterations.
It shows that the HZ method had much better performance than the existing methods.
Figure~\ref{fig4} plots the performance profile of each algorithm versus the elapsed time.
It also shows that the performance of the HZ method was much better than those of the existing methods.
In particular, the figures show that the HZ method is suitable for solving Problem \ref{pbl:stability}.

\begin{figure}[htbp]
\centering
 \includegraphics[scale=0.5]{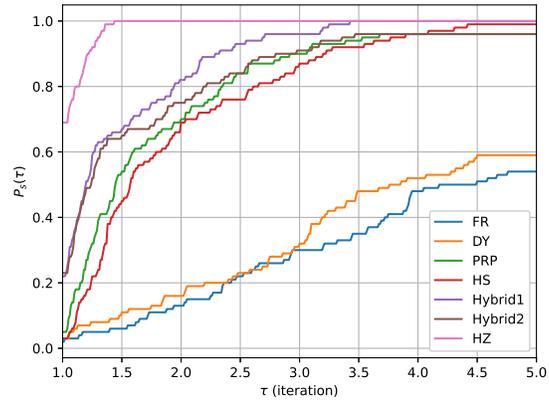}
\caption{Performance profile versus number of iterations for Problem \ref{pbl:stability}. \label{fig3}}
\end{figure}

\begin{figure}[htbp]
\centering
 \includegraphics[scale=0.5]{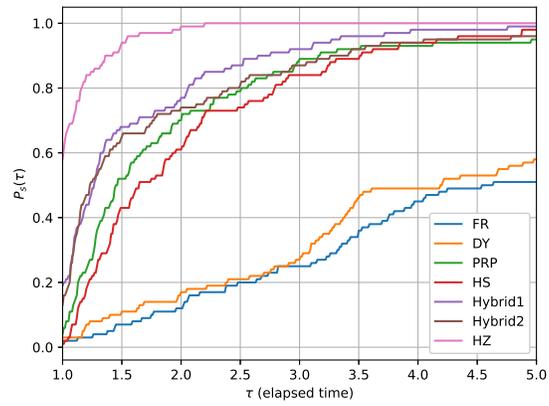}
\caption{Performance profile versus elapsed time for Problem \ref{pbl:stability}. \label{fig4}}
\end{figure}

\section{Conclusion}\label{sec:5}
We have presented a Hager--Zhang (HZ)-type Riemannian conjugate gradient method that uses exponential retraction and presented two global convergence properties under different assumptions.
We numerically compared the performance of the proposed method with those of existing Riemannian conjugate gradient methods for two Riemannian optimization problems on the unit sphere.
The results show that the HZ method has much better performance than the FR, DY, PRP, and HS methods.
In particular, we showed that the HZ method is suitable for the stability number of graphs problem.
It had much better performance than existing methods, including hybrid methods, for computing the stability number.

In a future paper, we will present the HZ method using a general retraction and its convergence analyses.

\bibliographystyle{abbrv}
\bibliography{biblib}

\end{document}